\documentclass{article}
\usepackage[utf8]{inputenc}
\usepackage{amssymb}
\usepackage{amsmath}
\usepackage{amsthm}
\usepackage[english]{babel}
\usepackage{graphicx}
\usepackage{bbm}
\usepackage[g]{esvect}
\usepackage{enumitem}
\usepackage{verbatim}
\usepackage{fullpage}
\usepackage{mathrsfs}
\usepackage{ifthen}
\usepackage{color}
\usepackage{psfrag}
\usepackage[unicode,colorlinks=true,pagebackref=false]{hyperref}

\numberwithin{equation}{section}
\theoremstyle{plain}
     \newtheorem{lemma}{Lemma}[section]
     \newtheorem{proposition}[lemma]{Proposition}
     \newtheorem{theorem}[lemma]{Theorem}
     \newtheorem{corollary}[lemma]{Corollary}
\theoremstyle{definition}
     
     \newtheorem{remark}[lemma]{Remark}
\theoremstyle{remark}

\renewenvironment{theorem}{\begin{Theorem}}{\end{Theorem}}

\renewenvironment{proof}{\begin{Proof}[\bfseries\proofname]}{\end{Proof}}
\newenvironment{prf}[1]{\begin{Proof}[\textbf{Proof of #1}]}{\end{Proof}}

\newcommand{\bbC}{\mathbb{C}}
\newcommand{\D}{\mathbb{D}}

\newcommand{\Cinf}{\overline{\mathbb{C}}}
\newcommand{\R}{\mathbb{R}}
\newcommand{\T}{\mathbb{T}}

\newcommand{\Hol}{\text{Hol}}

\title{Szeg\H o's theorem for Jordan arcs}
\author{Benedikt Buchecker\footnote{Department of Mathematics, KU Leuven, 3000 Leuven, Belgium, E-Mail: benedikt.buchecker@student.kuleuven.be}}
\date{September 15, 2025}

\begin{document}

\maketitle
\begin{abstract}
\noindent
The $n$-th Christoffel function for a point $z_0\in\mathbb C$ and a finite measure $\mu$ supported on a Jordan arc $\Gamma$ is 
\[
\lambda_n(\mu,z_0)=\inf\left\{\int_\Gamma |P|^2d\mu\mid P\text{ is a polynomial of degree at most }n\text{ and } P(z_0)=1\right\}.
\]
It is natural to extend this notion to $z_0=\infty$ and define $\lambda_n(\mu,\infty)$ to be the infimum of the squared $L^2(\mu)$-norm over monic polynomials of degree $n$. The classical Szeg\H o theorem provides an asymptotic description of $\lambda_n(\mu,z_0)$ for $|z_0|>1$ and $z_0=\infty$ and arbitrary finite measures supported on the unit circle. Widom has proved a version of Szeg\H o's theorem for measures supported on $C^{2+}$-Jordan arcs for the point $z_0=\infty$ and purely absolutely continuous measures belonging to the Szeg\H o class. We extend this result in two directions. We prove explicit asymptotics of $\lambda_n(\mu,z_0)$ for any finite measure $\mu$ supported on a $C^{1+}$-Jordan arc $\Gamma$, and for all points $z_0\in\mathbb{C}\cup\{\infty\}\setminus\Gamma$. Moreover, if the measure is in the Szeg\H o class, we provide explicit asymptotics for the extremal and orthogonal polynomials.
\end{abstract}

\section{Introduction}
We want to study asymptotics of orthogonal polynomials and other $L^2$-extremal polynomials as the degree tends to infinity. For fixed $n$, the extremal value is given by the value of the Christoffel function. For a positive and finite measure $\mu$ with compact and infinite support $K$ in $\bbC$ and $z_0\in \bbC$ the $n$-th \emph{Christoffel function} is given by
\begin{align}\label{intro1}
    \lambda_n(\mu,z_0)=\inf\left\{\int_K |P|^2d\mu\mid P\text{ is a polynomial of degree at most }n\text{ and } P(z_0)=1\right\}.
\end{align}
This can be extended to $z_0=\infty$ by setting
\begin{align}\label{intro2}
    \lambda_n(\mu,\infty) := \inf\left\{\int_K |P|^2 d\mu: P \text{ is a monic polynomial of degree } n\right\}.
\end{align}
Both minimization problems \eqref{intro1} and \eqref{intro2} admit a unique minimizer $P_{n,\mu,z_0}$ which we call the $n$-th \emph{minimizing polynomial} of the measure $\mu$. For $z_0=\infty$, the minimizing polynomial $P_{n,\mu,\infty}$ is the monic orthogonal polynomial associated with the measure $\mu$. Historically, the problems for $z_0=\infty$ and $z_0\in \bbC$ were studied independently. However, in recent works \cite{Eic17,ELY24,BEZ25} a unified approach proved successful. In \cite{BEZ25} it was shown that after appropriate rescaling the Christoffel function is continuous for all $z_0\in \Cinf$, where $\Cinf$ denotes the Riemann sphere $\bbC\cup \{\infty\}$. To highlight this unified approach, we denote the norm of the monic orthogonal polynomial associated with $\mu$ as $\lambda_n(\mu,\infty)$, differing from common notation.

Asymptotic results of the Christoffel function and minimizing polynomials date back to Bernstein, de la Valle Poussin, and Szeg\H o. Applications of these results, especially for orthogonal polynomials, can be found in very different settings such as random matrix theory \cite{Dei00}, spectral theory \cite{OPUC}, and data analysis \cite{LPP22}. 

The classical Szeg\H o theorem on the outside of the unit disc $\D$ states
\begin{align}\label{intro3}
    \lim_{n\to \infty} |z_0|^{-2n}\lambda_n(\mu,z_0)= (1-|z_0|^{-2}) \exp\left(\int_{\partial \D} \frac{|z_0|^{2}-1}{|\zeta-z_0|^2} \log f(\zeta) \,dm(\zeta)\right),
\end{align}
 where $|z_0|>1$, $\mu = f\,dm +\mu_s$ is a positive and finite measure supported on the unit circle $\partial \D$ with Lebesgue decomposition with respect to the normalized Lebesgue measure $dm$ on $\partial \D$. We note that the measures may have a singular part that does not appear on the right-hand side. Similar observations were made in \cite{BEZ25,CLW24,Spanish} where the singular part does not influence the asymptotics of the Christoffel function. Even so, the fact that the asymptotics can be computed for measures with a singular part proved important in \cite{BEZ25} for the study of $L^\infty$-extremal polynomials on a closed smooth Jordan curve. In this work, we investigate Jordan arcs, i.e., non-closed curves. 

For a Jordan arc, the difference to a closed Jordan curve is that there is no inner domain, and the boundary can be approached from two sides within the exterior domain. This makes any analysis more difficult, which is reflected in the fact that many problems that are solved for Jordan curves remain unsolved for Jordan arcs. One prominent example of this is the Chebyshev polynomials, which correspond to $L^\infty$-minimizers at the point $z_0=\infty$. While asymptotics of Chebyshev polynomials are known for a wide range of subsets of the real line \cite{AZ24x,CR24x,CSZ17,widom} and finite unions of Jordan curves \cite{widom}, for Jordan arcs the problem remains open except for circular arcs \cite{TD91,Eic17}, i.e., subarcs of the unit circle. A similar discrepancy is present in the $L^2$-case. In his seminal work \cite{widom}, Widom derived asymptotics of orthogonal polynomials, i.e., at the extremal point $z_0=\infty$, for purely absolutely continuous measures supported on the finite disjoint union of $C^{2+}$-Jordan arcs and curves. Asymptotics for arbitrary measures (that may have a singular part) are only known in two special cases for $z_0=\infty$: intervals \cite{NS91} or circular arcs \cite{Spanish}. We are able to settle the $L^2$-case by proving a complete analogue of \eqref{intro3} for a smooth Jordan arc. That is, we compute asymptotics of the Christoffel function and extremal polynomials for arbitrary measures supported on a $C^{1+}$-Jordan arc and any $z_0$ outside the Jordan arc.

Let us introduce the necessary quantities. Let $\Gamma$ be a $C^{1+\alpha}$-Jordan arc. That is, $\Gamma$ is a homeomorphic image of the unit interval with a parametrization that is continuously differentiable and the derivative satisfies a Lipschitz condition with exponent $\alpha>0$. If we do not want to specify the exponent, we will write $C^{1+}$. Let $\Omega = \Cinf\backslash \Gamma$ and $\Phi_{z_0}$ the conformal map from $\Omega$ to $\Cinf\setminus\overline{\D}$ such that $\Phi_{z_0}(\infty)=\infty$ and $\Phi_{z_0}(z_0)>0$ if $z_0\neq\infty$ and $\Phi_{z_0}'(z_0)>0$ if $z_0=\infty$. For $z_0\in\Omega$, we define 
\begin{align}\label{intro4}
    C(\Gamma,z_0)=\begin{cases}
        1/\Phi_{z_0}(z_0) &,z_0\ne \infty,\\
        1/\Phi_{\infty}'(\infty) &,z_0 =\infty.
    \end{cases}
\end{align}
Then we consider the $L^2$-\emph{Widom factor}
\begin{align}\label{intro5}
    W_{2,n}(\mu,z_0) := C(\Gamma,z_0)^{-n}\lambda_n(\mu,z_0)^{1/2}.
\end{align}
The $L^2$-Widom factors are usually considered only for the point $z_0=\infty$ where $C(\Gamma,\infty)$ is the classical logarithmic capacity of $\Gamma$. The above definition for other points $z_0$ is a continuous extension to $\Omega$, see \cite{BEZ25}. 

Furthermore, we need some basic notions of potential theory, which can be found in \cite{AG01,GM08,Hel14,Lan72,Ran95,ST97}. Let $\omega_{z_0}$ be the harmonic measure of the domain $\Omega$ and point $z_0\in \Omega$, and $\rho_{z_0}$ the density of $\omega_{z_0}$ with respect to the arc-length measure $|dz|$ on $\Gamma$.
\begin{theorem}\label{thm1}
    Let $\Gamma \in C^{1+}$ be a Jordan arc, $\mu= {f_{z_0}} d\omega_{z_0} + \mu_s$ a positive and finite measure with singular part $\mu_s$ with respect to the harmonic measure $\omega_{z_0}$ and $z_0\in \Cinf\backslash \Gamma$. Then,
    \begin{align}\label{eq:widom_limit}
        &\lim_{n\to\infty}W_{2,n}(\mu,z_0)^2 =(1-\Phi_{z_0}(z_0)^{-2}) \frac{2\pi}{|\Phi_{z_0}'(z_0)|}\exp\left(\int_\Gamma \log (f_{z_0}\rho_{z_0})d\omega_{z_0}\right).
    \end{align}    
\end{theorem}

We will also prove asymptotics of the minimizing polynomials. To state the result, we need to introduce the Szeg\H o function and Hardy space corresponding to a measure $\mu$. We say that a positive and finite measure $\mu = {f_{z_0}} d\omega_{z_0} + \mu_s$ where $\mu_s$ is singular with respect to $\omega_{z_0}$ satisfies the \emph{Szeg\H{o} condition} if
\begin{align}\label{szego_cond}
    \int_\Gamma \log {f_{z_0}} d\omega_{z_0} >-\infty.
\end{align}
Note that this is independent of $z_0$ and equivalent to 
\begin{align}\label{szego_alt}
    \int_\Gamma \log ({f_{z_0}} \rho_{z_0}) |dz| > -\infty,
\end{align}
where $\rho_{z_0}$ is the density of $\omega_{z_0}$ with respect to the arc-length measure $|dz|$ on $\Gamma$. The condition \eqref{szego_alt} is the Szeg\H o condition as introduced by Widom \cite{widom}. As will become clear from the formulation of Corollary \ref{cor2}, the harmonic measure in \eqref{szego_cond} is more natural in our setting; see also \cite{Alp22}. If $\mu$ satisfies the Szeg\H o condition \eqref{szego_cond} we define the \emph{Szeg\H o function} of the density $f_{z_0}$
\begin{align}\label{szego_func}
    R_{{f_{z_0}}}(z) =  \exp\left(\int \log {f_{z_0}}\,d\omega_z + i* \int \log {f_{z_0}}\,d\omega_z\right)
\end{align}
where $i* \int \log {f_{z_0}}\,d\omega_z$ is the harmonic conjugate of the harmonic function with boundary values $\log {f_{z_0}}$ such that $R_{f_{z_0}}(z_0)>0$. Then $R_{f_{z_0}}$ is analytic and non-zero in $\Omega$ with $|R_{f_{z_0}}(z)|={f_{z_0}}(z)$ on the boundary $\omega_{z_0}$ a.e.. If $\mu$ does not satisfy the Szeg\H o condition, then we set $R_{f_{z_0}}=0$. 

Let $H^2(\D)$ be the standard Hardy space on the unit disc. Then we define
\begin{align*}
    H^2(\Omega) := \left\{ F \in \Hol(\Omega): z\mapsto F(\Phi_{z_0}^{-1}(1/z)) \in H^2(\D)\right\},
\end{align*}
where $\Hol(\Omega)$ is the set of all holomorphic functions $F:\Omega\to \bbC$.
If $\mu$ satisfies the Szeg\H{o} condition then the corresponding Hardy space is defined by
\begin{align*}
    H^2(\Omega,\mu) := \left\{ F \in \Hol(\Omega): F \sqrt{R_{f_{z_0}}} \in H^2(\Omega)\right\}.
\end{align*}
As we will see in Section 2, every $F\in H^2(\Omega,\mu)$ has boundary values $F_+,F_-$ at $|dz|$-a.e. point on both sides of $\Gamma$ with $F_+,F_- \in L^2({f_{z_0}}d\omega_{z_0})$, and 
\begin{align*}
    \|F\|_{H^2(\Omega,\mu)} := \left(\oint_\Gamma |F|^2 {f_{z_0}}d\omega_{z_0} \right)^{1/2}= \left(\int_\Gamma (|F_+|^2 + |F_-|^2) {f_{z_0}}d\omega_{z_0} \right)^{1/2}
\end{align*}
defines a norm on $H^2(\Omega,\mu)$. It is even a reproducing kernel Hilbert space (see also \cite{widom}). Thus, we can consider the following quantity
\begin{align}\label{intro8}
    \nu(\mu,z_0) = \inf\left\{ \oint_\Gamma |F|^2 {f_{z_0}}d\omega_{z_0} : F\in H^2(\Omega,\mu),F(z_0)=1\right\}.
\end{align}
If $K_\mu$ is the reproducing kernel of $H^2(\Omega,\mu)$, the Fischer-Riesz theorem implies that $\nu(\mu,z_0)= \frac{1}{K_\mu(z_0,z_0)}$ and that $F_{\mu,z_0}= \frac{K_\mu(\cdot,z_0)}{K_\mu(z_0,z_0)}$ is the unique minimizer of \eqref{intro8}. Now, we can state the main result on the asymptotics of the minimizing polynomials.
\begin{theorem}\label{thm3}
    Let $\Gamma \in C^{1+}$ be a Jordan arc, $\mu={f_{z_0}}d\omega_{z_0} + \mu_s$ a positive and finite measure with singular part $\mu_s$ and $z_0\in \Cinf\backslash \Gamma$. If $\mu$ satisfies the Szeg\H{o} condition and $P_{n,\mu,z_0}$ is the $n$-th minimizing (monic orthogonal) polynomial for $\mu$ then
    \begin{align*}
        & |C(\Gamma,z_0)|^{-2n}\lambda_n(\mu,z_0)\xrightarrow{n\to\infty} \nu(\mu,z_0),\\
        &\int_\Gamma |C(\Gamma,z_0)^{-n}P_{n,\mu,z_0} - H_n|^2 {f_{z_0}}d\omega_{z_0}\xrightarrow{n\to\infty} 0,\\
        &\frac{C(\Gamma,z_0)^{-n}}{\Phi_{z_0}(z)^n}P_{n,\mu,z_0}(z) \xrightarrow{n\to\infty} F_{\mu,z_0}(z)
    \end{align*}
     \text{locally uniform in $\Omega$}, where
    \begin{align*}
        H_n(z)&= \Phi_{z_0}^+(z)^n F_{\mu,z_0}^+(z) + \Phi_{z_0}^-(z)^n F_{\mu,z_0}^-(z).
    \end{align*}
\end{theorem}
It is conjectured in \cite{BEZ25} that the limits of $L^\infty$- and $L^2$-Widom factors for the harmonic measure are closely related and are given by the reproducing kernel associated to the harmonic measure $\omega_{z_0}$. The extremal property in Corollary \ref{cor2} gives a hint at the distinct role played by the harmonic measure.
\begin{corollary}\label{cor2}
    Let $\Gamma \in C^{1+}$ be a Jordan arc, $\mu={f_{z_0}}d\omega_{z_0} + \mu_s$ a positive and finite measure with singular part $\mu_s$ and $z_0\in \Cinf\backslash \Gamma$. Then
    \begin{align*}
        \lim_{n\to\infty} W_{2,n}(\mu,z_0)^2 = \frac{1}{K_{\omega_{z_0}}(z_0,z_0)} \exp\left(\int_\Gamma \log f_{z_0}\,d\omega_{z_0}\right).
    \end{align*}
    Furthermore, the harmonic measure $\omega_{z_0}$ is the unique maximizer of $\lim_{n\to\infty}W_{2,n}(\mu,z_0)$ over all probability measures $\mu$ supported on $\Gamma$.
\end{corollary}

In Section 2, we provide necessary estimates and prerequisites for the proofs of Theorems \ref{thm1} and \ref{thm3} in Section 3. This includes a discussion of the Hardy space $H^2(\Omega,\mu)$.

\subsection*{Acknowledgment}
This research was carried out at TU Wien, supported by project P33885-N of the Austrian Science Fund FWF. I want to thank Benjamin Eichinger for giving me the opportunity to work on this project and supporting me in the writing of this paper.

\section{Estimates}
First, we show an explicit formula for the density of the harmonic measure using the conformal map $\Phi_{z_0}$. 
An important fact from Widom \cite[Section 11]{widom} is that we can 'open up' the arc. Let $A, B$ be the endpoints of $\Gamma$. Then we define the conformal map 
\begin{align*}
    \varphi^{-1}(z)=\frac{2z-A-B}{B-A} + \frac{2}{B-A}\sqrt{(z-A)(z-B)},
\end{align*}
where the branch of the square root is such that $\sqrt{(z-A)(z-B)}$ is asymptotically $z$ as $z\to\infty$. Then $\varphi^{-1}$ maps $\Omega$ to the outside $\Omega'$ of a Jordan curve $\Gamma'$ where $0$ lies in the interior. Lemma 11.1 in \cite{widom} shows that $\Gamma'$ is a $C^{1+} $ Jordan curve. Furthermore,  
 \begin{align}\label{eq:20}
    (\varphi^{-1})'(z)= \frac{\varphi^{-1}}{\sqrt{(z-A)(z-B)}},
\end{align}
and the function $\varphi$ is also explicitly given by 
\begin{align*}
    \varphi(s)= ((B-A)(s+s^{-1})/2 +B+A)/2.
\end{align*} 
The square root $\sqrt{(z-A)(z-B)}$ has different boundary values depending on the side of $\Gamma$ and thus $\varphi^{-1}$ extends continuously differentiable to the boundary $\Gamma$ as $\varphi^{-1}_+$ and $\varphi^{-1}_-$. They satisfy $\varphi^{-1}_+ = 1/\varphi^{-1}_-$ because $\varphi(s)=\varphi(t) \iff s= t \lor s=t^{-1}$.

The conformal map $\Phi_{z_0}\circ \varphi$ extends continuously to a nonzero $C^{1+}$-function on $\Gamma'$ by \cite{garnett}. Then, $\Phi_{z_0}$ extends continuously differentiable to $\Gamma$ without the endpoints as $\Phi_{z_0}^\pm = (\Phi_{z_0}\circ \varphi) \circ\varphi^{-1}_\pm$. We can also immediately see from the formula of the derivative of $\varphi^{-1}$ that 
\begin{align*}
    |(\Phi_{z_0}^{\pm})'|\sqrt{|\cdot-A||\cdot-B|}
\end{align*}
defines a nonzero $C^{0+}$-function on $\Gamma$, see also \cite[Appendix A]{totik}. 

For $z_0\in\Omega$, let $B_{z_0}$ denote the conformal map from $\Omega$ to $\Cinf\setminus\overline{\D}$ such that $B_{z_0}(z_0)=\infty$ and $B_{z_0}'(\infty)>0$. Note that $\Phi_\infty= B_\infty $ and $\Phi_{z_0}=\frac{|B_{\infty}(z_0)|}{B_\infty(z_0)}B_\infty$ for $z_0\ne\infty$. The harmonic measure $\omega_{z_0}$ of the domain $\Omega$ and point $z_0\in \Omega$ is then given by 
\begin{align}\label{eq:1}
    \omega_{z_0}= \frac{1}{2\pi}\left(|(B_{z_0}^{+})'(z)| + |(B_{z_0}^{-})'(z)|\right) |dz|
\end{align}
where $|dz|$ is the arc-length measure on $\Gamma$, and $(B_{z_0}^{\pm})'(z)$ denote the two boundary values of the function $B_{z_0}'(\cdot)$. For $z_0=\infty$, this is well known, see e.g. \cite{Alp22}, and for arbitrary $z_0\in\Omega$, this follows directly by conformal invariance of all related quantities. Let $\rho_{z_0}$ be the density of $\omega_{z_0}$ with respect to $|dz|$.
Note that there exists an unimodular constant $c$ so that 
    \begin{align}\label{eq:3}
        B_{z_0}(z)=c\frac{1-\Phi_{z_0}(z)\Phi_{z_0}(z_0)}{\Phi_{z_0}(z)-\Phi_{z_0}(z_0)}.
    \end{align}
    Then,
    \begin{align}\label{eq:4}
        B_{z_0}'(z)=c\frac{(\Phi_{z_0}(z_0)^2-1)\Phi'_{z_0}(z)}{(\Phi_{z_0}(z)-\Phi_{z_0}(z_0))^2}.
    \end{align}
Then, \eqref{eq:1} allows us to give a formula of $\rho_{z_0}$ in terms of $\Phi_{z_0}$
\begin{align*}
    \rho_{z_0} = \frac{\Phi_{z_0}(z_0)^2-1}{2\pi} \left(\frac{|(\Phi_{z_0}^+)'(z)|}{\left|\Phi_{z_0}^+(z)- \Phi_{z_0}(z_0)\right|^2} + \frac{|(\Phi_{z_0}^-)'(z)|}{\left|\Phi_{z_0}^-(z)- \Phi_{z_0}(z_0)\right|^2}\right).
\end{align*}
Since $ |(\Phi_{z_0}^\pm)'(z)|\sqrt{|z-A||z-B|}$ is a positive $C^{0+}$-function on $\Gamma$ and $0<c_1<\left|\Phi_{z_0}^\pm(z)- \Phi_{z_0}(z_0)\right|<c_2<\infty$, we know that $\rho_{z_0}\sqrt{|z-A||z-B|}$ is also positive and $C^{0+}$ on $\Gamma$.

We want to prove the properties of $H^2(\Omega,\mu)$.
\begin{proposition}\label{prop3.2}
    Let $\Gamma$ be a $C^{1+}$-Jordan arc. Let $\mu = {f_{z_0}} d\omega_{z_0} + \mu_s$ be a positive and finite measure supported on $\Gamma$ that satisfies the Szeg\H o condition. We define $R_\mu= R_{f_{z_0}}R_{\rho_{z_0}}$. Then, $H^2(\Omega,\mu)$ is a reproducing kernel Hilbert space with reproducing kernel
    \begin{align}\label{rep_kern}
        K_\mu(z,w)= \frac{1}{2\pi}\sqrt{\frac{\Phi_{z_0}'(z)}{R_\mu(z)}}\overline{\sqrt{\frac{\Phi_{z_0}'(w)}{R_\mu(w)}}}\frac{1}{1- 1/(\Phi_{z_0}(z)\overline{\Phi_{z_0}(w)})}.
    \end{align}
    Every $F\in H^2(\Omega,\mu)$ has $|dz|$-a.e. boundary values $F_+,F_-$ on both sides of $\Gamma$ with $F_+,F_- \in L^2({f_{z_0}}d\omega_{z_0})$ and the norm is given by 
\begin{align}\label{h2_norm}
    \|F\|_{H^2(\Omega,\mu)} = \left(\oint_\Gamma |F|^2 {f_{z_0}}d\omega_{z_0} \right)^{1/2}.
\end{align}
\end{proposition}
\begin{proof}
    We have seen that 
    \begin{align*}
        \rho_{z_0}\sqrt{|z-A||z-B|}, |(\Phi_{z_0}^\pm)'|\sqrt{|z-A||z-B|} 
    \end{align*}
    are $C^{0+}$-functions on $\Gamma$. Then the maximum principle shows that 
    \begin{align*}
        \frac{\sqrt{R_{\rho_{z_0}}}}{\sqrt{\Phi_{z_0}'}} \text{ and }\frac{\sqrt{\Phi_{z_0}'}}{\sqrt{R_{\rho_{z_0}}}}
    \end{align*}
    are holomorphic and bounded in $\Omega$. Then 
    \begin{align*}
        F \in H^2(\Omega,\mu) &\iff z \mapsto \sqrt{R_{f_{z_0}}(\Phi_{z_0}^{-1}(1/z))}F(\Phi_{z_0}^{-1}(1/z)) \in H^2(\D) \\
        &\iff z\mapsto \sqrt{\frac{R_\mu(\Phi_{z_0}^{-1}(1/z))}{\Phi_{z_0}'(\Phi_{z_0}^{-1}(1/z))}}F(\Phi_{z_0}^{-1}(1/z)) \in H^2(\D) 
    \end{align*}
    Therefore we can rewrite the definition of $H^2(\Omega,\mu)$ to 
    \begin{align}
        H^2(\Omega,\mu ) = \left\{ \frac{1}{\sqrt{2\pi}}\sqrt{\frac{\Phi_{z_0}'}{R_\mu}}  G\circ (1/\Phi_{z_0}):G\in H^2(\D) \right\}.
    \end{align}
    If we introduce a scalar product on $H^2(\Omega,\mu)$ such that the bijective map
    \begin{align*}
        \kappa: H^2(\D)\ni G \mapsto \frac{1}{\sqrt{2\pi}}\sqrt{\frac{\Phi_{z_0}'}{R_\mu}}  G\circ (1/\Phi_{z_0}) \in H^2(\Omega,\mu)
    \end{align*}
    becomes an isometry, then $H^2(\Omega,\mu)$ becomes a reproducing kernel Hilbert space with reproducing kernel
    \begin{align*}
        K_\mu(z,w)= \frac{1}{2\pi}\sqrt{\frac{\Phi_{z_0}'(z)}{R_\mu(z)}}\overline{\sqrt{\frac{\Phi_{z_0}'(w)}{R_\mu(w)}}}\frac{1}{1- 1/(\Phi_{z_0}(z)\overline{\Phi_{z_0}(w)})}
    \end{align*}
    because $1/(1-z\overline{w})$ is the reproducing kernel of $H^2(\D)$. A simple calculation shows that formula \eqref{h2_norm} follows from the definition of $\kappa$ and the fact that $\Phi_{z_0}^\pm$ is a diffeomorphism from $\Gamma$ onto its image. 
\end{proof}
We can also characterize the reproducing kernel using another outer function instead of $R_\mu$, which simplifies the formula for $F_{\mu,z_0}$.
\begin{lemma}
 Let $\Gamma$ be a $C^{1+}$-Jordan arc. Let $\mu = {f_{z_0}} d\omega_{z_0} + \mu_s$ be a positive and finite measure supported on $\Gamma$ that satisfies the Szeg\H o condition. Then,
\[
R(z)=\frac{R_{\rho_{z_0}}(z)B_{z_0}(z)^2}{\Phi_{z_0}(z)^2B_{z_0}'(z)}
\]
is the outer function with boundary values $|R_\pm(z)|={\rho_{z_0}(z)}/{|(B_{z_0}^{\pm})'(z)|}$ on $\Gamma$. Moreover,
\begin{align}\label{rep_kern_2}
     K_{\mu}(z,w)=\frac{1}{2\pi}\frac{1}{\sqrt{R(z)R_{f_{z_0}}(z)}}\frac{1}{\overline{\sqrt{R(w)R_{f_{z_0}}(w)}}}\frac{(\Phi_{z_0}(z)\Phi_{z_0}(z_0)-1)(\Phi_{z_0}(z_0)\overline{\Phi_{z_0}(w)}-1)}{(\Phi_{z_0}(z_0)^2-1)(\Phi_{z_0}(z)\overline{\Phi_{z_0}(w)}-1)}
\end{align}
\end{lemma}
\begin{proof}
    We know that $\Phi'_{z_0}$ is analytic and non-vanishing in $\Omega$. From \eqref{eq:3} and \eqref{eq:4}, it follows that the same holds for 
    \begin{align}\label{eq:5}
        \frac{B_{z_0}'(z)\Phi_{z_0}(z)^2}{B_{z_0}(z)^2}=\frac{(\Phi_{z_0}(z_0)^2-1)\Phi'_{z_0}(z)\Phi_{z_0}(z)^2}{(1-\Phi_{z_0}(z_0)\Phi_{z_0}(z))^2},
    \end{align}
    Moreover, for $z\in\Gamma\setminus\{A,B\}$ we have
    \[
    \left|\frac{R_{\rho_{z_0}}(z)B_{z_0}^\pm(z)^2}{(B_{z_0}^{\pm})'(z)\Phi_{z_0}^\pm(z)^2}\right|=\frac{\rho_{z_0}(z)}{|(B_{z_0}^{\pm})'(z)|}.
    \]
    Rearranging \eqref{eq:5} gives
    \[
    \Phi'_{z_0}(z)=\frac{(1-\Phi_{z_0}(z_0)\Phi_{z_0}(z))^2}{(\Phi_{z_0}(z_0)^2-1)\Phi_{z_0}(z)^2}\frac{B_{z_0}'(z)\Phi_{z_0}(z)^2}{B_{z_0}(z)^2},
    \]
    and
    \begin{align*}
        \frac{\Phi_{z_0}'(z)}{R_{\rho_{z_0}}(z)R_{f_{z_0}}(z)}=\frac{(1-\Phi_{z_0}(z_0)\Phi_{z_0}(z))^2}{(\Phi_{z_0}(z_0)^2-1)\Phi_{z_0}(z)^2}\frac{B_{z_0}'(z)\Phi_{z_0}(z)^2}{B_{z_0}(z)^2R_{\rho_{z_0}}(z)R_{f_{z_0}}(z)}=\frac{(1-\Phi_{z_0}(z_0)\Phi_{z_0}(z))^2}{(\Phi_{z_0}(z_0)^2-1)\Phi_{z_0}(z)^2}\frac{1}{R(z)R_{f_{z_0}(z)}}.
    \end{align*}
    Using \eqref{rep_kern}, we conclude that 
    \begin{align*}
        K_{\mu}(z,w)=\frac{1}{2\pi}\frac{1}{\sqrt{R(z)R_{f_{z_0}}(z)}}\frac{1}{\overline{\sqrt{R(w)R_{f_{z_0}}(w)}}}\frac{(\Phi_{z_0}(z)\Phi_{z_0}(z_0)-1)(\Phi_{z_0}(z_0)\overline{\Phi_{z_0}(w)}-1)}{(\Phi_{z_0}(z_0)^2-1)(\Phi_{z_0}(z)\overline{\Phi_{z_0}(w)}-1)}
    \end{align*}
\end{proof}
\begin{remark}\label{rem1}Let $\mu=f_{z_0}d\omega_{z_0}+\mu_s$ be a positive and finite measure supported on $\Gamma$ satisfying the Szeg\H o condition. Then we note two things:
\begin{enumerate}
    \item Since $R_{f_{z_0}} \circ \Phi_{z_0}^{-1}(1/\cdot)$ is the outer function in $\D$ with boundary values $f_{z_0} \circ \Phi_{z_0}^{-1}(1/\cdot)$ on $\T$ we know that $\sqrt{R_{f_{z_0}}}\in H^2(\Omega)$ and therefore $H^\infty(\Omega)\subseteq H^2(\Omega,\mu)$.
    \item From formulas \eqref{rep_kern} and \eqref{rep_kern_2} we get
    \begin{align}\label{formula_nu}
        \nu(\mu,z_0) &=\frac{1}{K_\mu(z_0,z_0)}= 2\pi (1-\Phi_{z_0}(z_0)^{-2}) \frac{R_\mu(z_0)}{|\Phi_{z_0}'(z_0)|}=\frac{R_{f_{z_0}}(z_0)}{K_{\omega_{z_0}}(z_0,z_0)}=2\pi R_{f_{z_0}}(z_0) R(z_0).
    \end{align}
    Formula \eqref{rep_kern_2} also proves an explicit formula for the extremal function $F_{\mu,z_0}$
    \begin{align}\label{formula_F_mu}
        F_{\mu,z_0} (z) &= \frac{K_\mu(z,z_0)}{K_\mu(z_0,z_0)}=\sqrt{\frac{R_{f_{z_0}}(z_0)R(z_0)}{R_{f_{z_0}}(z)R(z)}}.
    \end{align}
\end{enumerate}
\end{remark}

We will later use the fact that $F_{\mu,z_0} \circ \varphi$ extends continuously to a $C^{0+}$-function on the transformed Jordan curve $\Gamma'$.
\begin{lemma}\label{lem1}
    Let $\Gamma \in C^{1+}$ be a Jordan arc, $\mu={f_{z_0}}d\omega_{z_0} + \mu_s$ a positive and finite measure with singular part $\mu_s$ and $z_0\in \Cinf\backslash \Gamma$. If ${f_{z_0}}$ is nonzero and $C^{0+}$ on $\Gamma$, then 
    \begin{enumerate}[label=(\roman*)]
        \item $F_{\mu,z_0} \circ\varphi$ is in $C^{0+}$ on $\Gamma'$,
        \item $|F_{\mu,z_0}(z)|^2\leq C\frac{1}{{f_{z_0}}(z)}$ for $z\in\Gamma$ where $C>0$ is independent of the measure $\mu$.
    \end{enumerate}
\end{lemma}
\begin{proof}
The assumption on the density implies that $\mu$ satisfies the Szeg\H o condition. Then, (ii) follows directly from \eqref{formula_F_mu} and $R_{f_{z_0}} = f_{z_0}$ on $\Gamma$. Furthermore,
\begin{align*}
    F_{\mu,z_0}(\varphi(s))^2 = C \frac{\Phi_{z_0}(\varphi(s))^2 B_{z_0}'(\varphi(s))}{B_{z_0}(\varphi(s))^2 R_{f_{z_0}}(\varphi(s))R_{\rho_{z_0}}(\varphi(s))} = C \frac{(\Phi_{z_0}\circ\varphi)(s)^2 (B_{z_0}\circ\varphi)'(s)}{(B_{z_0}\circ\varphi)(s)^2 } \frac{1}{R_{f_{z_0}}(\varphi(s))R_{\rho_{z_0}}(\varphi(s))\varphi'(s)}
\end{align*}
with a constant $C\in \bbC\backslash \{0\}$. $\Phi_{z_0}\circ \varphi$ and $B_{z_0}\circ \varphi$ are conformal maps of the outside of $\Gamma'$ to the outside of the unit circle. By Chapters I and II in \cite{garnett}, they extend continuously to a $C^{1+}$-function on $\Gamma'$. A straightforward calculation using \eqref{eq:20} shows
\begin{align*}
    |R_{f_{z_0}}(\varphi(s))R_{\rho_{z_0}}(\varphi(s))\varphi'(s)|= {f_{z_0}}(\varphi(s))\rho_{z_0}(\varphi(s)) |\varphi'(s)|= \left(\left({f_{z_0}} \rho_{z_0} \sqrt{|\cdot-A||\cdot-B|} \right)\circ \varphi \right)(s) \frac{1}{|s|}
\end{align*}
for $s\in \Gamma'$. We know that $\rho_{z_0}\sqrt{|\cdot-A||\cdot-B|}$ is nonzero and $C^{0+}$ on $\Gamma$ and ${f_{z_0}}$ satisfies the same by assumption. 
Then, Lemma 11.1 and Lemma 4.1. (4) in \cite{widom} show that $R_{f_{z_0}}(\varphi(s))R_{\rho_{z_0}}(\varphi(s))\varphi'$ extends continuously to $\Gamma'$, is nonzero and in $C^{0+}$ on $\Gamma'$ which proves (i). 
\end{proof}
To prove Theorem \ref{thm1}, we will first prove a lower bound for measures satisfying the Szeg\H{o} condition. Here, the singular part provides no problems. 
\begin{proposition}\label{prop1}
    Let $\Gamma \in C^{1+}$ be a Jordan arc, $\mu={f_{z_0}}d\omega_{z_0} + \mu_s$ a positive and finite measure with singular part $\mu_s$ and $z_0\in \Cinf\backslash \Gamma$. If $\mu$ satisfies the Szeg\H{o} condition then
    \begin{align*}
        \liminf_{n\to \infty} W_{2,n}(\mu,z_0)^2 \geq \frac{1}{K_\mu(z_0,z_0)}.
    \end{align*}
\end{proposition}
\begin{proof}
    The case $z_0=\infty$ is Lemma 12.2. in \cite{widom} using that $W_{2,n}(\mu,\infty)\geq W_{2,n}({f_{\infty}}d\omega_{\infty},\infty)$. Therefore assume $z_0\ne \infty$.
    Let $p\in \mathcal{P}_n$ be a polynomial. Then 
    \begin{align*}
        \Phi_{z_0}^{-n}p \in \Hol(\Omega).
    \end{align*}
    Because of the maximum principle, it is even in $H^\infty(\Omega)$. Remark \ref{rem1} shows $\Phi_{z_0}^{-n}p\in H^2(\Omega,\mu)$ and we see
    \begin{align}\label{eq1}
        p(z_0) \Phi_{z_0}(z_0)^{-n} = \oint_\Gamma p(z) \Phi_{z_0}(z)^{-n} \overline{K_\mu(z,z_0)} {f_{z_0}}d\omega_{z_0}(z)= \int_\Gamma p(z) \overline{L_n(z,z_0)} {f_{z_0}}d\omega_{z_0}(z)
    \end{align}
    where
    \begin{align*}
        L_n(z,z_0)=\Phi_{z_0}^+(z)^n K_\mu^+(z,z_0) + \Phi_{z_0}^-(z_0)^n K_\mu^-(z,z_0).
    \end{align*}
    Then Lemma 12.1 in \cite{widom} shows
    \begin{align*}
        \lim_{n\to\infty}\int_\Gamma |L_n(z,z_0)|^2 {f_{z_0}}d\omega_{z_0}(z) = K_\mu(z_0,z_0).
    \end{align*}
    If we further assume $p(z_0)=1$ then (\ref{eq1}) shows
    \begin{align*}
        \int_\Gamma |p|^2 d\mu \geq \int_\Gamma |p|^2 {f_{z_0}}d\omega_{z_0} \geq \frac{\Phi_{z_0}(z_0)^{-2n}}{\int_\Gamma |L_n(z,z_0)|^2 {f_{z_0}}d\omega_{z_0}}.
    \end{align*}
    This proves
    \begin{align*}
        \liminf_{n\to \infty} W_{2,n}(\mu,z_0)^2 \geq \frac{1}{K_\mu(z_0,z_0)}.
    \end{align*}
\end{proof}
For an upper estimate, it is good to have certain 'test' polynomials. In our case, the generalized Faber polynomials turn out to be the right choice.
\begin{lemma}\label{lem2}
    Let $\Gamma\in C^{1+\alpha}$ be a Jordan arc, $0<\alpha< 1$ and $F\in H^\infty(\Omega)$. Suppose that $F\circ \varphi$ extends continuously to a function in $C^{0+\alpha}$ on $\Gamma'$. We define the generalized Faber polynomial
    \begin{align*}
        T(w) = \frac{1}{2\pi i}\int_{\Gamma_R}\frac{F(z) \Phi_{z_0}(z)^n}{z-w}dz  
    \end{align*}
    with $\Gamma_R = \{z\in \bbC: |\Phi_{z_0}(z)|=R\}$ and for $w\in \bbC$ on the inside of $\Gamma_R$. Then, $T$ is a polynomial of degree $n$ with leading coefficient $(\Phi_{z_0}'(\infty))^n$. Furthermore,
    \begin{align*}
        T(w)= F_+(w)\Phi_{z_0}^+(w)^n+F_-(w)\Phi_{z_0}^-(w)^n + \mathcal{O}\left(\frac{\log n}{n^{\alpha}}\right)
    \end{align*}
    uniformly on $\Gamma$ and
    \begin{align*}
        T(w) = F(w) \Phi_{z_0}(w)^n + \mathcal{O}\left(\frac{\log n}{n^{\alpha}}\right)
    \end{align*}
    uniformly for $w$ in closed subsets of $\Omega$.
\end{lemma}
\begin{proof}
    For $w\in \overline{\Omega}$ we find $s\in \overline{\Omega'}$ such that $w= \varphi(s)$. With the transformation $\sigma = \varphi^{-1}(z)$ (see \cite[Lemma 11.2]{widom}) we get
    \begin{align*}
        T(w)=\frac{1}{2\pi i}\int_{\Gamma'_R} F(\varphi(\sigma)) \Phi_{z_0}(\varphi(\sigma))^n \frac{1}{\sigma-s}d\sigma +\frac{1}{2\pi i}\int_{\Gamma'_R} F(\varphi(\sigma)) \Phi_{z_0}(\varphi(\sigma))^n \frac{1}{\sigma s(\sigma-s^{-1})}d\sigma.
    \end{align*}
    These integrals are just the generalized Faber polynomials associated with the Jordan curve $\Gamma'$ and weights $F(\varphi), F(\varphi)/\text{id}$ as defined in \cite{suetin}. If $w\in \Gamma$ we know that $s,s^{-1}\in \Gamma'$ and they correspond to the two boundary values of $\varphi^{-1}_\pm$. Because of the assumptions, Theorem 2 of Chapter 4 Section 2 in \cite{suetin} applies, and we have
    \begin{align*}
        T(w) &= F(\varphi(s)) \Phi_{z_0}(\varphi(s))^n + F(\varphi(s^{-1}))\Phi_{z_0}(\varphi(s))^n +\mathcal{O}\left(\frac{\log n}{n^{\alpha}}\right)\\
        &= F_+(w)\Phi_{z_0}^+(w)^n+F_-(w)\Phi_{z_0}^-(w)^n + \mathcal{O}\left(\frac{\log n}{n^{\alpha}}\right).
    \end{align*}
    If we have a closed subset $L\subseteq \Omega$, then $1/\varphi^{-1}(L)$ is a compact subset of the interior of $\Gamma'$. Theorem 1 of Chapter 4 Section 1 in \cite{suetin} tells us that for $w\in L$
    \begin{align*}
        T(w) &= F(\varphi(s)) \Phi_{z_0}(\varphi(s))^n +\mathcal{O}\left(\frac{\log n}{n^{\alpha}}\right) + \mathcal{O}\left(\frac{1}{n^{\alpha}}\right)= F(w)\Phi_{z_0}(w)^n +\mathcal{O}\left(\frac{\log n}{n^{\alpha}}\right)
    \end{align*}
    which proves the result.
\end{proof}

We want to recall a result featured in \cite{OPUC}, which is used in the proof of Szeg\H o's theorem on the unit circle.
\begin{theorem}\label{thm2}\textnormal{\cite[Theorem 2.5.1. p.152]{OPUC}}
    Let $d\nu$ be a singular measure on $\partial \D$. Then, there exists a sequence of polynomials $P_n$ so that
    \begin{enumerate}
        \item[(i)] $\sup_{z\in\partial \D,n} |P_n(z)|=1$
        \item[(ii)] For $d\theta$-a.e. $e^{i\theta}\in \partial \D$,
        \begin{align*}
            \lim_{n\to\infty} |P_n(e^{i\theta})|=0
        \end{align*}
        \item[(iii)] Uniformly for $\zeta$ in compact subsets of $\D$,
        \begin{align*}
            \lim_{n\to\infty} |P_n(\zeta)|=0
        \end{align*}
        \item[(iv)] For $d\nu$-a.e. $\theta$,
        \begin{align*}
            \lim_{n\to\infty } P_n(e^{i\theta})=1
        \end{align*}
    \end{enumerate}
\end{theorem}
The proof of Theorem \ref{thm1} relies on the fact that we can disregard the singular part using these polynomials. However, we have to transform them to our domain $\Omega$.
\begin{corollary}\label{cor1}
    Let $\Gamma \in C^{1+}$ be a Jordan arc, $\mu={f_{z_0}}d\omega_{z_0} + \mu_s$ a positive and finite measure with singular part $\mu_s$ and $z_0\in \Cinf\backslash \Gamma$. Then we can construct $Q_k \in H_\infty(\Omega)$ such that 
    \begin{enumerate}
        \item[(i)] $Q_k$ are uniformly bounded on $\Omega$ and have boundary values everywhere on $\Gamma$,
        \item[(ii)] $Q_k(z_0)=1$,
        \item[(iii)] $Q_k^{\pm} \to 1$ $d\omega_{z_0}$-a.e.,
        \item[(iv)] $Q_k^{\pm} \to 0$ $\mu_s$-a.e., and
        \item[(v)] $Q_k\circ \varphi$ is $C^{1+}$ on $\Gamma'$.
    \end{enumerate}
\end{corollary}
\begin{proof}
    Let $\Gamma_0$ be $\Gamma$ without the endpoints. Then $\Phi_{z_0}^+$ and $\Phi_{z_0}^-$ are diffeomorphisms from $\Gamma_0$ to $\Phi_{z_0}^+(\Gamma_0)$ and $\Phi_{z_0}^-(\Gamma_0)$ respectively. Thus $\nu=(\Phi_{z_0}^+)^*\mu_s + (\Phi_{z_0}^-)^*\mu_s$ is singular on $\partial \D$. Let $ P_k$ be the polynomials associated with $\nu$ from Theorem \ref{thm2}. We define
    \begin{align*}
        R_k(\zeta) = (1-\overline{P_k(1/\overline{\zeta})})/(1-\overline{P_k(1/\Phi_{z_0}(z_0))}).
    \end{align*}
    Then $R_k\in \Hol(\Cinf\backslash \D)$ is uniformly bounded since $P_k(1/\Phi_{z_0}(z_0))\to 0$ and $|P_k|\leq 1$. We also know that $R_k(\Phi_{z_0}(z_0))=1$ and $R_k \in C^{2+}$ on $\partial \D$. Now we define $Q_k= R_k \circ\Phi_{z_0}$. Then properties $(i)$ and $(ii)$ are clear. To see $(iii)$ and $(iv)$ we note that $R_k(\zeta)\to 1$ $d\zeta$-a.e. on $\partial \D$ and $R_k(\zeta)\to 0$ $d\nu$-a.e. on $\partial \D$. Then 
    \begin{align*}
        \{z\in \Gamma: Q_k^+(z)\not\to 1\} = (\Phi_{z_0}^+)^{-1} (\{\zeta\in \partial \D: R_k(\zeta)\not\to 1\}).
    \end{align*}
    Since the set on the right has an arc-length measure of 0 on $\partial \D$ and the harmonic measure is conformal invariant, the set on the left has a harmonic measure of 0. We also have
    \begin{align*}
        \mu_s(\{z\in\Gamma: Q_k^+(z)\not\to 0 \})= \mu_s((\Phi_{z_0}^+)^{-1}\{\zeta\in\partial \D:R_k(\zeta)\not\to 0\})= (\Phi_{z_0}^+)^*\mu_s(\{\zeta\in\partial \D:R_k(\zeta)\not\to 0\})=0.
    \end{align*}
    The result for $Q_k^-$ follows similarly. For $(v)$ we consider $Q_k\circ \varphi= R_k \circ(\Phi_{z_0}\circ \varphi)$. $\Phi_{z_0}\circ \varphi$ is the conformal mapping from the unbounded component of $\Cinf\backslash \Gamma'$ to $\Cinf\backslash \D$ which is $C^{1+}$ because $\Gamma'$ is $C^{1+}$. Thus $Q_k\circ\varphi$ is the composition of two functions in $C^{1+}$ on $\Gamma'$ which shows $(v)$.
\end{proof}
\section{Proofs}
Now, we can prove our main theorems.
\begin{prf}{Theorem \ref{thm1}}
    The proof splits into 3 Steps.\\
        1. Step: Suppose ${f_{z_0}}\in C^{0+}$ and nonzero. Then according to Lemma \ref{lem1} (i) the function $F_{\mu,z_0}\circ \varphi$ is $C^{0+}$ on $\Gamma'$ and, in particular, in $H^\infty(\Omega)$. Let $Q_k$ be the sequence constructed in Corollary \ref{cor1}. Thus, $F_{\mu,z_0} \cdot Q_k\in H^\infty(\Omega)$ satisfies the conditions of Lemma \ref{lem2}. Then, the generalized Faber polynomial $T_{n,k}$ associated with the function $F_{\mu,z_0}\cdot Q_k$ have the property that 
        \begin{align*}
            T_{n,k}(z) - (F_{\mu,z_0}^+(z) Q_k^+(z)\Phi_{z_0}^+(z)^n + F_{\mu,z_0}^-(z) Q_k^-(z)\Phi_{z_0}^-(z)^n )\xrightarrow{n\to\infty} 0
        \end{align*}
        for every $k$ and uniform on $\Gamma$. They also satisfy $T_{n,k}(z_0)/ \Phi_{z_0}(z_0)^n \xrightarrow{n\to \infty} 1$ for $z_0\ne \infty$ and every $k$. If $z_0=\infty$, they have the leading coefficient $C(\Gamma,\infty)^n$. Thus 
        \begin{align*}
            W_{2,n}(\mu,z_0)^2 \leq \frac{\Phi_{z_0}(z_0)^{2n}}{|T_{n,k}(z_0)|^2} \int_\Gamma |T_{n,k}(z)|^2 d\mu(z)
        \end{align*}
        for $z_0\ne \infty$ and 
        \begin{align*}
            W_{2,n}(\mu,\infty)^2 \leq \int_\Gamma |T_{n,k}(z)|^2 d\mu(z)
        \end{align*}
        otherwise. Therefore, we have to evaluate this integral. Let 
        \begin{align*}
            H_n(z)&= \Phi_{z_0}^+(z)^n F_{\mu,z_0}^+(z) + \Phi_{z_0}^-(z)^n F_{\mu,z_0}^-(z), \\
            G_{n,k}(z)&= F_{\mu,z_0}^+(z) Q_k^+(z)\Phi_{z_0}^+(z)^n + F_{\mu,z_0}^-(z) Q_k^-(z)\Phi_{z_0}^-(z)^n,
        \end{align*}
        for $z\in \Gamma$. Then,
        \begin{align*}
            \int_\Gamma |T_{n,k}(z)|^2 d\mu(z) -\nu(\mu,z_0) &=
            \int_\Gamma \left(|T_{n,k}(z)|^2 -|G_{n,k}(z)|^2\right) d\mu(z) \\
            &+ \int_\Gamma |G_{n,k}(z)|^2 d\mu_s(z)\\
            &+ \int_\Gamma \left(|G_{n,k}(z)|^2- |H_n(z,z_0)|^2\right) {f_{z_0}}(z)d\omega_{z_0}(z)\\
            &+ \left(\int_\Gamma |H_n(z,z_0)|^2 {f_{z_0}}d\omega_{z_0}(z)-K_\mu(z_0,z_0)^{-1}\right)
        \end{align*}
        The first term vanishes for $n\to\infty$ for every $k$. Since $F_{\mu,z_0}(z)= K(z,z_0)/K(z_0,z_0)$ and $F_{\mu,z_0} \Phi_{z_0}^n$ are uniformly bounded on $\Gamma$ we know that the second and third term tend to zero for $k\to\infty$ uniformly in $n$ which follows from expanding the modulus and properties $(iii)$ and $(iv)$ in Corollary \ref{cor1}. For the fourth term, note
        \begin{align*}
            \int_\Gamma |H_n(z,z_0)|^2 {f_{z_0}}d\omega_{z_0}(z)-K_\mu(z_0,z_0)^{-1} &= \oint_\Gamma |F_{\mu,z_0}|^2 f_{z_0}d\omega_{z_0} \\
            &+ 2 \text{Re}\left(\int_\Gamma \Phi_{z_0}^+(z)^n F_{\mu,z_0}^+ \Phi_{z_0}^-(z)^{-n}\overline{F_{\mu,z_0}^-(z,z_0)}f_{z_0}(z)d\omega_{z_0}(z) \right) \\
            &- K_\mu(z_0,z_0)^{-1},
        \end{align*}
        where the first and third term cancel each other and the second approaches zero for $n\to\infty$ according to Lemma 12.1 in \cite{widom}. All in all, we get
        \begin{align*}
            \limsup_{n\to \infty }W_{2,n}(\mu,z_0)^2 \leq \nu(\mu,z_0).
        \end{align*}
        2. Step: We only assume that ${f_{z_0}}$ is bounded below. Let $\sigma = g d\omega_{z_0}+\mu_s$ where $g\in C^{0+}$ and nonzero. Let $T_{n,k}$, $G_{n,k}$ and $H_n$ be as before but associated with $\sigma$. Again we have
        \begin{align*}
                W_{2,n}(\mu,z_0)^2 \leq \begin{cases}
                    \frac{\Phi_{z_0}(z_0)^{2n}}{|T_{n,k}(z_0)|^2} \int_\Gamma |T_{n,k}|^2 d\mu &,z_0\ne \infty\\
                    \int_\Gamma |T_{n,k}|^2 d\mu &,z_0=\infty.
                \end{cases}
        \end{align*}
        Furthermore,
        \begin{align*}
            \int_\Gamma |T_{n,k}|^2 d\mu&= \int_\Gamma |T_{n,k}|^2 -|G_{n,k}|^2d\mu \\
            &+ \int_\Gamma |G_{n,k}|^2 d\mu_s\\
            &+ \int_\Gamma |G_{n,k}|^2- |H_n(\cdot,z_0)|^2 {f_{z_0}} \,d \omega_{z_0}\\
            &+ \int_\Gamma |H_n(\cdot,z_0)|^2({f_{z_0}}-g)d\omega_{z_0} \\
            &+ \int_\Gamma |H_n(\cdot,z_0)|^2g d\omega_{z_0}.
        \end{align*}
        The first term vanishes for every $k$ and $g$ if $n\to \infty$. The second and third term tend to zero for $k\to \infty$ uniform in $n$ and $g$ if $g^{-1}$ is bounded uniformly by Lemma \ref{lem1} (ii). The limit for $n\to\infty$ of the fifth term is $\nu(\sigma,z_0)$ for every $g$. For the fourth term we can construct a sequence $g_l$ such that $g_l^{-1}$ is bounded and
        \begin{align*}
            \int_\Gamma |{f_{z_0}}-g_l| d\omega_{z_0} \to 0.
        \end{align*}
        Then, the fourth term tends to zero uniformly in $n$ as $l\to\infty$. We also get
        \begin{align*}
            \int_\Gamma |\log {f_{z_0}} -\log g_l| d\omega_{z_0} \to 0.
        \end{align*}
        Using \eqref{szego_func} and \eqref{formula_nu}, this implies
        \begin{align*}
            \nu(\sigma_l,z_0) \to \nu(\mu,z_0).
        \end{align*}
        This shows
        \begin{align*}
            \limsup_{n\to\infty} W_{2,n}(\mu,z_0)^2 \leq \nu(\mu,z_0).
        \end{align*}
        3. Step: Now, we drop all assumptions. So let $\mu={f_{z_0}}\omega_{z_0}+\mu_s$ be any positive and finite measure supported on $\Gamma$. Then $f_\epsilon(z) = {f_{z_0}}(z)+\epsilon$ for $\epsilon>0$ is bounded from below and satisfies the Szeg\H{o}-condition. Let $\mu_\epsilon = f_\epsilon d\omega_{z_0}+\mu_s$. From the second step and Remark \ref{rem1}, we know
        \begin{align*}
            \limsup_{n\to \infty} W_{2,n}(\mu,z_0)^2 \leq \limsup_{n\to \infty} W_{2,n}(\mu_\epsilon,z_0)^2\leq \frac{R_{f_\epsilon}(z_0)}{K_{\omega_{z_0}}(z_0,z_0)}. 
        \end{align*}
        Since
        \begin{align*}
            R_{f_\epsilon}(z_0) = \exp\left(\int_\Gamma \log({f_{z_0}}+\epsilon)d\omega_{z_0}\right) \xrightarrow{\epsilon\to 0 } R_{f_{z_0}}(z_0),
        \end{align*}
        because of monotonous convergence, we know
        \begin{align*}
            \limsup_{n\to \infty} W_{2,n}(\mu,z_0)^2 \leq \frac{R_{f_{z_0}}(z_0)}{K_{\omega_{z_0}}(z_0,z_0)}.
        \end{align*}
        Combined with Proposition \ref{prop1} (if ${f_{z_0}}$ satisfies the Szeg\H{o}-condition) and Remark \ref{rem1} we get
        \begin{align*}
            \lim_{n\to\infty}W_{2,n}(\mu,z_0)^2 = \frac{R_{f_{z_0}}(z_0)}{K_{\omega_{z_0}}(z_0,z_0)}= (1-\Phi_{z_0}(z_0)^{-2}) \frac{2\pi}{|\Phi_{z_0}'(z_0)|}R_\mu(z_0).
        \end{align*}
\end{prf}
\begin{prf}{Theorem \ref{thm3}}
        We want to show the strong asymptotics provided that $\mu$ satisfies the Szeg\H{o}-condition. In proving Theorem \ref{thm1} and with Remark \ref{rem1} we have already shown the first assertion
        \begin{align*}
             \lim_{n\to\infty}|C(\Gamma,z_0)|^{-2n}\lambda_n(\mu,z_0) = \nu(\mu,z_0).
        \end{align*}
        It remains to prove the asymptotics of the minimizing polynomials. Using the reproducing kernel property of $K_\mu(\cdot,\cdot)$ and $H_n(z) = (\Phi_{z_0}^+(z)^n K_{\mu}^+(z,z_0) + \Phi_{z_0}^-(z_0)^n K_{\mu}^-(z,z_0)) /K_\mu(z_0,z_0)$ we get 
        \begin{align*}
            \int_\Gamma |C(\Gamma,z_0)^{-n}P_{n,\mu,z_0} - H_n|^2 {f_{z_0}}d\omega_{z_0} &\leq W_{2,n}(\mu,z_0) + \int_\Gamma |H_n|^2 {f_{z_0}}d\omega_{z_0} \\
						&-2 \text{Re}\left(\int_\Gamma C(\Gamma,z_0)^{-n}P_{n,\mu,z_0} \overline{H_n}{f_{z_0}}d\omega_{z_0}\right)\\
            &= W_{2,n}(\mu,z_0) + \int_\Gamma |H_n|^2 {f_{z_0}}d\omega_{z_0} \\
            &-2\nu(\mu,z_0)\text{Re}\left(\oint_\Gamma C(\Gamma,z_0)^{-n}P_{n,\mu,z_0}(z) \Phi_{z_0}(z)^{-n} \overline{K_\mu(z,z_0)}{f_{z_0}}d\omega_{z_0}\right)\\
            &= \nu(\mu,z_0) + \nu(\mu,z_0) -2\nu(\mu,z_0) + o(1),
        \end{align*}
        as $n\to\infty$, which proves the second assertion. Furthermore, 
        \begin{align*}
            \frac{C(\Gamma,z_0)^{-n}P_{n,\mu,z_0}(w)}{\Phi_{z_0}(w)^n}&= \oint_\Gamma \frac{C(\Gamma,z_0)^{-n}P_{n,\mu,z_0}(z)}{\Phi_{z_0}(z)^n}\overline{K_\mu(z,w)}{f_{z_0}}(z)d\omega_{z_0}(z) \\
            &=\oint_\Gamma \left(\frac{C(\Gamma,z_0)^{-n}P_{n,\mu,z_0}(z)}{\Phi_{z_0}(z)^n}-\frac{H_n(z)}{\Phi_{z_0}(z)^n}\right)\overline{K_\mu(z,w)}{f_{z_0}}(z)d\omega_{z_0}(z)\\
            &+\oint_\Gamma \frac{H_n(z)}{\Phi_{z_0}(z)^n}\overline{K_\mu(z,w)}{f_{z_0}}(z)d\omega_{z_0}(z).
        \end{align*}
        Based on what we have shown before, the first term approaches zero for $n\to\infty$. The convergence is even uniformly for $w$ in closed subsets of $\Omega$ since $K_\mu(w,w)$ is continuous in $\Omega$. The second evaluates to
        \begin{align*}
            \oint_\Gamma \frac{H_n(z)}{\Phi_{z_0}(z)^n}\overline{K_\mu(z,w)}{f_{z_0}}(z)d\omega_{z_0}(z) &= \oint_\Gamma F_{\mu,z_0}(z)\overline{K_\mu(z,w)} {f_{z_0}}(z)d\omega_{z_0}(z) \\
            &+\int_\Gamma \Phi_{z_0}^-(z)^n F_{\mu,z_0}^-(z)\Phi_{z_0}^+(z)^{-n}\overline{K_\mu^+(z,w)}{f_{z_0}}(z)d\omega_{z_0}(z)\\
            &+\int_\Gamma \Phi_{z_0}^+(z)^n F_{\mu,z_0}^+(z)\Phi_{z_0}^-(z)^{-n}\overline{K_\mu^-(z,w)}{f_{z_0}}(z)d\omega_{z_0}(z).
        \end{align*}
        Here, the first term is $F_{\mu,z_0}(w)$. Lemma 12.1. in \cite{widom} shows that the second and third term vanish uniformly for $w$ in closed subsets of $\Omega$. 
\end{prf}
\begin{prf}{Corollary \ref{cor2}}
    Theorem \ref{thm1} and Remark \ref{rem1} show
    \begin{align*}
        \lim_{n\to\infty} W_{2,n}(\mu,z_0)^2=\frac{R_{f_{z_0}}(z_0)}{K_{\omega_{z_0}}(z_0,z_0)} = \frac{1}{K_{\omega_{z_0}}(z_0,z_0)} \exp\left(\int_\Gamma \log f_{z_0}d\omega_{z_0}\right).
    \end{align*}
    We see that, in order to maximize $\lim_{n\to\infty}W_{2,n}(\mu,z_0)$ over probability measures we have to maximize $R_{f_{z_0}}(z_0)$. Thus we may assume that $\mu= {f_{z_0}}d\omega_{z_0} + \mu_s$ satisfies the Szeg\H{o} condition. Then 
    \begin{align*}
        R_{f_{z_0}}(z_0) = \exp\left(\int_\Gamma \log {f_{z_0}} d\omega_{z_0} \right)\leq \exp\left(\log(\int_\Gamma {f_{z_0}}d\omega_{z_0}) \right)=\mu(\Gamma)-\mu_s(\Gamma)= 1-\mu_s(\Gamma).
    \end{align*}
    The first inequality follows from the Jensen inequality with equality if and only if ${f_{z_0}}$ is constant. Thus, the unique maximizer is ${f_{z_0}}\equiv 1$, corresponding to the harmonic measure $\omega_{z_0}$.
\end{prf}

%%%%%%%%%%%%%%%%%%%%%%%%%%


\begin{thebibliography}{99}
{}
\bibitem{Alp22}
G. Alpan. Extremal polynomials on a Jordan arc. \emph{J. Approx. Theory} 276 (2022), Paper
No. 105708, 8.
{}
\bibitem{AZ24x}
G. Alpan and M. Zinchenko. Lower Bounds for Weighted Chebyshev and Orthogonal Polynomials.
2024. arXiv: \href {https://arxiv.org/abs/2408.11496} {\nolinkurl {2408.11496}
\texttt{[math.CA]}}.
{}
\bibitem{AG01}
D. H. Armitage and S. J. Gardiner. Classical potential theory. Springer Monographs in Mathematics.
Springer-Verlag London, Ltd., London, 2001, pp. xvi+333.
{}
\bibitem{Spanish}
M. Bello Hern\'andez and E. Mi\~na D\'iaz. Strong asymptotic behavior and weak convergence of
polynomials orthogonal on an arc of the unit circle. \emph{J. Approx. Theory} 111.2 (2001),
pp. 233–255.
{}
\bibitem{BEZ25}
B. Buchecker, B. Eichinger, and M. Zinchenko. Asymptotics of $L^r$ extremal polynomials for
$0<r\leq \infty$ on $C^{1+}$ Jordan regions. 2025. arXiv: \href
{https://arxiv.org/abs/2502.17616} {\nolinkurl {2502.17616} \texttt{[math.CA]}}.
{}
\bibitem{CLW24}
S. Charpentier, N. Levenberg, and F. Wielonsky. An extremal problem for the –B˝ergman kernel of
orthogonal polynomials. \emph{Constr. Approx.} 61.1 (2025), pp. 63–80.
{}
\bibitem{CR24x}
J. S. Christiansen and O. Rubin. Chebyshev polynomials related to Jacobi weights. 2024. arXiv: \href
{https://arxiv.org/abs/2409.02623} {\nolinkurl {2409.02623} \texttt{[math.CA]}}.
{}
\bibitem{CSZ17}
J. S. Christiansen, B. Simon, and M. Zinchenko. Asymptotics of Chebyshev polynomials, I:
subsets of $\R$. \emph{Invent. Math.} 208.1 (2017), pp. 217–245.
{}
\bibitem{Dei00}
P. Deift. Orthogonal Polynomials and Random Matrices: A Riemann-Hilbert Approach. Courant
lecture notes in mathematics. Courant Institute of Mathematical Sciences, New York University,
2000.
{}
\bibitem{Eic17}
B. Eichinger. Szeg\H o-Widom asymptotics of Chebyshev polynomials on circular arcs. \emph{J.
Approx. Theory} 217 (2017), pp. 15–25.
{}
\bibitem{ELY24}
B. Eichinger, M. Luki\'c, and G. Young. Asymptotics of Chebyshev rational functions with respect
to subsets of the real line. \emph{Constr. Approx.} 59.3 (2024), pp. 541–581.
{}
\bibitem{GM08}
J. B. Garnett and D. E. Marshall. Harmonic measure. Vol. 2. New Mathematical Monographs.
Reprint of the 2005 original. Cambridge University Press, Cambridge, 2008, pp. xvi+571.
{}
\bibitem{garnett}
John B. Garnett and Donald E. Marshall. Harmonic measure. Vol. 2. New Mathematical
Monographs. Cambridge University Press, Cambridge, 2005, pp. xvi+571.
{}
\bibitem{Hel14}
L. L. Helms. Potential theory. Second. Universitext. Springer, London, 2014, pp. xiv+485.
{}
\bibitem{Lan72}
N. S. Landkof. Foundations of modern potential theory. Vol. Band 180. Die Grundlehren der
mathematischen Wissenschaften. Translated from the Russian by A. P. Doohovskoy. Springer-Verlag,
New York-Heidelberg, 1972, pp. x+424.
{}
\bibitem{LPP22}
J. B. Lasserre, E. Pauwels, and M. Putinar. The Christoffel–Darboux Kernel for Data Analysis.
Cambridge Monographs on Applied and Computational Mathematics. Cambridge University Press,
2022.
{}
\bibitem{NS91}
E. M. Nikishin and V. N. Sorokin. Rational approximations and orthogonality. eng. Translations of
mathematical monographs ; v. 92. Providence, R.I: American Mathematical Society, 1991.
{}
\bibitem{Ran95}
T. Ransford. Potential theory in the complex plane. Vol. 28. London Mathematical Society Student
Texts. Cambridge University Press, Cambridge, 1995, pp. x+232.
{}
\bibitem{ST97}
E. B. Saff and V. Totik. Logarithmic potentials with external fields. Vol. 316. Grundlehren der
mathematischen Wissenschaften [Fundamental Principles of Mathematical Sciences]. Appendix B by
T. Bloom. Springer-Verlag, Berlin, 1997, pp. xvi+505.
{}
\bibitem{OPUC}
B. Simon. Orthogonal polynomials on the unit circle. Part 1. Vol. 54, Part 1. American
Mathematical Society Colloquium Publications. Classical theory. American Mathematical Society,
Providence, RI, 2005, pp. xxvi+466.
{}
\bibitem{suetin}
P. K. Suetin. Series of Faber polynomials. Vol. 1. Analytical Methods and Special Functions.
Translated from the 1984 Russian original by E. V. Pankratiev. Gordon and Breach Science
Publishers, Amsterdam, 1998, pp. xx+301.
{}
\bibitem{TD91}
J.-P. Thiran and C. Detaille. Chebyshev polynomials on circular arcs in the complex plane.
\emph{Progress in approximation theory}. Academic Press, Boston, MA, 1991, pp. 771–786.
{}
\bibitem{totik}
V. Totik. Asymptotics of Christoffel functions on arcs and curves. \emph{Adv. Math.} 252 (2014),
pp. 114–149.
{}
\bibitem{widom}
H. Widom. Extremal polynomials associated with a system of curves in the complex plane.
\emph{Advances in Math.} 3 (1969), pp. 127–232.
\end{thebibliography}
\end{document}